\documentclass{amsart}
\usepackage{latexsym,amsmath,amssymb}
\newtheorem{thm}{Theorem}[section]
\newtheorem{cor}[thm]{Corollary}
\newtheorem{exam}[thm]{Example}

\newtheorem{prop}[thm]{Proposition}
\theoremstyle{definition}\newtheorem{defn}[thm]{Definition}
\theoremstyle{remark}

\numberwithin{equation}{section}

\begin{document}

\title[Weighted conditional type operators]{Quasi-contractivity, Stability and convergence of WCT operators
}

\author{\sc\bf Y. Estaremi and Z. Huang}
\address{\sc Y. Estaremi}
\email{y.estaremi@gu.ac.ir}
\address{Department of Mathematics, Faculty of Sciences, Golestan University, Gorgan, Iran.}
\address{ \sc Z. Huang}
\email{jameszhuang923@gmail.com}
\address{Huxley Building Department of Mathematics, South Kensington Campus,Imperial College London, London, UK}
\thanks{}



\date{}

\dedicatory{}

\commby{}

\begin{abstract}
In this paper we characterize quasi-contraction, stable and convergent weighted conditional type (WCT) operators on $L^p(\mu)$. Indeed we provide equivalent conditions for quasi-contraction WCT operators. Also, we prove that convergence, uniformly stability, strongly stability and weakly stability of WCT operators are equivalent. Finally we provided some concrete examples to illustrate our main results.
\end{abstract}

\maketitle

\section{ \sc\bf Introduction and Preliminaries}
Let $\mathcal{H}$ be a complex Hilbert spaces, $\mathcal{B}(\mathcal{H})$ be the Banach algebra of all bounded linear operators on $\mathcal{H}$, where $I=I_{\mathcal{H}}$ is the identity operator on $\mathcal{H}$. If $T\in \mathcal{B}(\mathcal{H})$, then $T^*$ is the adjoint of $T$.

Let $A, T\in \mathcal{B}(\mathcal{H})$, where $A$ is non-zero positive operator. The operator $T$ is called $A$-contraction if $T^*AT\leq A$.
It is easy to see that if $T$ is $A$-contraction, then $T^n$ is also $A$-contraction, for every $n\in \mathbb{N}$. In order to \cite{cs}, for $n\in \mathbb{N}$, we say that $T$ is  $n$-quasi-contraction if $T$ is $T^{*^n}T^n$-contraction, i.e., $T^{*^{n+1}}T^{n+1}\leq T^{*^{n}}T^n$. Hence $T$ is a $n$-quasi-contraction if and only if $T$ is a contraction on $\mathcal{R}(T^n)$. Moreover, $T$ is called quasi-contraction if it is 1-quasi-contraction.\\

Let $(X, \mathcal{F}, \mu)$ be a $\sigma$-finite measure space.
For a $\sigma$-subalgebra $\mathcal{A}$ of $\mathcal{F}$, the conditional expectation operator associated with $\mathcal{A}$ is the mapping $f \rightarrow E^{\mathcal{A}}f$, defined for all non-negative measurable function $f$ as well as for all $f \in L^p(\mathcal{F})=L^p(X, \mathcal{F}, \mu)$, in which $1\leq p<\infty$. The function $E^{\mathcal{A}}f$ is the unique $\mathcal{A}$-measurable function that satisfies the equation:

$$\int_{A}(E^{\mathcal{A}}f)d\mu = \int_{A}fd\mu, \ \ \ \ \ \ \ \forall A\in \mathcal{A}.$$

We will often use the notation $E$ instead of $E^{\mathcal{A}}$. This operator will play a significant role in our
 work, and we list some of its useful properties here:\\

\noindent $\bullet$ \  If $g$ is
$\mathcal{A}$-measurable, then $E(fg)=E(f)g$, for all $f\in L^p(\mathcal{F})$.

\noindent $\bullet$ \ If $f\geq 0$, then $E(f)\geq 0$; if $E(|f|)=0$,
then $f=0$.

\noindent $\bullet$ \ $|E(fg)|\leq
(E(|f|^p))^{\frac{1}{p}}(E(|g|^{q}))^{\frac{1}{q}}$, where $\frac{1}{p}+\frac{1}{q}=1$, $f\in L^p(\mathcal{F})$ and $f\in L^q(\mathcal{F})$.

\noindent $\bullet$ \ For each $f\geq 0$, $S(E(f))$ is the smallest $\mathcal{A}$-measurable set containing $S(f)$, where $S(f)=\{x\in X: f(x)\neq 0\}$.

A detailed discussion and verifications of
most of these properties may be found in \cite{rao}.

Combination of conditional expectation operator $E$ and multiplication operators
appears more often in the service of the study of other operators such as multiplication operators and weighted composition operators. Specifically, in \cite{mo}, S.-T. C.
Moy has characterized all operators on $L^p(\mu)$ of the form $f\rightarrow E(fg)$, for $g\in L^q(\mu)$ with $E(|g|)$ bounded. In \cite{dou}, R. G. Douglas analyzed positive projections on $L^1(\mu)$ and many of his characterizations are in terms of combinations of multiplications and conditional expectations. Weighted conditional type operators are studied by many mathematicians recent years in \cite{ye,e1,ej, ej1, es, gd, her} and references therein. \\

 Here we recall the definition of weighted conditional type operators on $L^p$-spaces.

\begin{defn}
Let $(X,\mathcal{F},\mu)$ be a $\sigma$-finite measure space and $\mathcal{A}$ be a
$\sigma$-sub-algebra of $\mathcal{F}$ such that $(X,\mathcal{A},\mu_{\mathcal{A}})$ is also $\sigma$-finite. Let $E$ be the conditional
expectation operator relative to $\mathcal{A}$. If $u,w:X\rightarrow \mathbb{C}$ are $\mathcal{F}$-measurable functions such that $uf$ is conditionable (i.e., $E(uf)$ exists) and $wE(uf)\in L^{p}(\mathcal{F})$ for all $f\in L^{p}(\mathcal{F})$, where $1\leq p<\infty$, then the corresponding weighted conditional type (or briefly WCT) operator is the linear transformation $M_wEM_u:L^p(\mathcal{F})\rightarrow L^{p}(\mathcal{F})$ defined by $f\rightarrow wE(uf)$.
\end{defn}

In this paper we characterize quasi-contraction, stable and convergent WCT operators on $L^p(\mu)$. Indeed we prove that a WCT operator is quasi-contraction iff it is $n$-quasi-contraction iff it is power bounded and we provide equivalent conditions for quasi-contraction WCT operators. Also, we prove that convergence, uniformly stability, strongly stability and weakly stability of WCT operators are equivalent. Finally we provided some concrete examples to illustrate our main results.

\section{ \sc\bf Quasicontraction, Stable and convergent WCT operators }

%
In this section we first characterize quasi-contraction WCT operators. Then the relation between quasi-contractivity and power boundednes are investigated. In the sequel we find some equivalent conditions for uniformly, strongly and weakly stability of WCT operators. Moreover, convergent WCT operators are characterized.\\

Let $X$ be a Banach space and $T\in \mathcal{B}(X)$. The operator $T$ is called power bounded if $\sup_n\|T^n\|<\infty$.
In \cite{ye}, Theorem 2.5, part (b) there is an error that we correct it in the next Theorem, indeed we should replace the symbol $<$ by $\leq$.


\begin{thm}\label{t3.2} Let $T=M_wEM_u\in\mathcal{B}(L^p(\mathcal{F}))$, $1\leq p<\infty$ and $\frac{1}{p}+\frac{1}{p'}=1$. Then $T$ is power bounded if and only if $|E(wu)|\leq1$ on $S((E(|w|^p))^{\frac{1}{p}}(E(|u|^{p'}))^{\frac{1}{p'}})$.
\end{thm}
\begin{proof} Let $|E(wu)|\leq1$ on $S((E(|w|^p))^{\frac{1}{p}}(E(|u|^{p'}))^{\frac{1}{p'}})$. Since
$$S(E(uw))\subseteq S((E(|w|^p))^{\frac{1}{p}}(E(|u|^{p'}))^{\frac{1}{p'}}),$$
 then $\|E(wu)\|_{\infty}\leq1$ on $X$. Hence $\|(E(w))^n\|_{\infty}=\|E(wu)\|^n_{\infty}\leq 1$ and so the sequence  $\{\|E(uw)^n\|_{\infty}\}_{n\in \mathbb{N}}$ is uniformly bounded. Therefore there exists $C>0$ such that $\|E(uw)^n\|_{\infty}\leq C$, for all $n\in \mathbb{N}$. Moreover, $T^n=M_{E(uw)^{n-1}}T$. Hence for every $f\in L^p(\mathcal{F})$ and  $n\in \mathbb{N}$ we have
$$\|T^nf\|_p=\|E(uw)^{n-1}T(f)\|_p\leq \|E(uw)^{n-1}\|_{\infty}\|T\|\|f\|_p\leq C\|T\|\|f\|_p.$$
This means that $T$ is power bounded.\\

Conversely, let $T$ be power bounded. Then we can find $C>0$ such that
$$\|M_{(E(uw))^{n-1}}T\|\leq C, \text{for all} n\in \mathbb{N}.$$
 We know that
  $$S(Tf)\subseteq S((E(|w|^p))^{\frac{1}{p}}(E(|u|^{p'}))^{\frac{1}{p'}}) \ \ \text{and} \ \ S(E(uw))\subseteq S((E(|w|^p))^{\frac{1}{p}}(E(|u|^{p'}))^{\frac{1}{p'}}).$$
   Suppose that there exists $A\in \mathcal{F}$ with $\mu(A)>0$ such that $|E(uw)|>1$ on $A$. Then $\|E(uw)\|_{\infty}>1$ and so $\|(E(uw))^n\|_{\infty}\rightarrow \infty$. In this case $M_{(E(uw))^{n-1}}T$ is bounded if and only if $T=0$. So if $T\neq 0$, then we get a contradiction. Therefore we should have $|E(wu)|\leq1$ on $S((E(|w|^p))^{\frac{1}{p}}(E(|u|^{p'}))^{\frac{1}{p'}})$.
\end{proof}

Let $T\in \mathcal{B}(\mathcal{H})$, $n\in \mathbb{N}$. Then $T$ is $n$-quasi-contraction if and only if
$$\|T^nx\|^2\geq\|T^{n+1}x\|^2, \forall x\in \mathcal{H},$$
since
$$T^{*^{n+1}}T^{n+1}\leq T^{*^n}T^n \Leftrightarrow T^{*^n}T^n\geq T^{*^{n+1}}T^{n+1}$$
$$ \Leftrightarrow \langle T^{*^n}T^nx,x\rangle-\langle T^{*^{n+1}}T^{n+1}x,x\rangle\geq 0 \ \ \ \forall x\in \mathcal{H}$$
$$\Leftrightarrow \|T^nx\|^2-\|T^{n+1}x\|^2\geq 0, \forall x\in \mathcal{H}.$$
$$ \Leftrightarrow \|T^nx\|^2\geq\|T^{n+1}x\|^2, \forall x\in \mathcal{H}.$$
Now in the next theorem we characterize $n$-quasi-contraction WCT operators on $L^p$-spaces.
\begin{thm}\label{t3.3}
  Let $T=M_wEM_u\in \mathcal{B}(L^2(\mu))$ and $n\in \mathbb{N}$. Then $T$ is $n$-quasi-contraction if and only if $T$ is 1-quasi-contraction if and only if $|E(uw)|\leq 1$, $\mu$, a.e.
 \end{thm}
 \begin{proof}
  Then for each $n\in \mathbb{N}$,
$$T^n=M_{E(uw)^{n-1}}M_wEM_u=M_{E(uw)^{n-1}}T$$
and similarly
$$ T^{*^{n}}=M_{E(\bar{u}\bar{w})^{n-1}}M_{\bar{u}}EM_{\bar{w}}=M_{E(\bar{u}\bar{w})^{n-1}}T^*.$$
Hence
$$T^{*^{n}}T^n=M_{|E(uw)|^{2(n-1)}}T^*T=M_{E(|w|^2)|E(uw)|^{2(n-1)}}M_{\bar{u}}EM_u.$$
So
$$T^{*^n}T^n\geq T^{*^{n+1}}T^{n+1}$$
if and only if
$$M_{E(|w|^2)|E(uw)|^{2(n-1)}}M_{\bar{u}}EM_u\geq M_{E(|w|^2)|E(uw)|^{2(n)}}M_{\bar{u}}EM_u$$
if and only if
$$M_{E(|w|^2)|E(uw)|^{2(n-1)}}M_{(1-|E(uw)|^2)}M_{\bar{u}}EM_u\geq 0.$$
Let $T_0=M_{E(|w|^2)|E(uw)|^{2(n-1)}}$, $T_1=M_{(1-|E(uw)|^2)}$ and $T_2=M_{\bar{u}}EM_u$. Then it is easy to see that $T_0\geq 0$, $T_2\geq 0$ and $$T_0T_1T_2=T_1T_0T_2=T_2T_0T_1=T_1T_2T_0=T_0T_2T_1.$$

 Therefore $T_0T_1T_2\geq0$ if and only if $T_1\geq 0$ if and only if $1-|E(uw)|^2\geq 0$, $\mu$, a.e., if and only if $|E(uw)|\leq 1$, $\mu$, a.e., on $F$.\\
By these observations we get that for each $n\in \mathbb{N}$, $$T^{*^n}T^n\geq T^{*^{n+1}}T^{n+1}$$ if and only if $$T^{*}T\geq T^{*^{2}}T^{2}$$ if and only if $|E(uw)|\leq 1$, $\mu$, a.e., on $F$.
\end{proof}

In the following Theorem we find that a WCT operators is $n$-quasi-contraction if and only if it is power bounded. Also, some other equivalent conditions are provided.

\begin{thm}\label{t3.4}
Let $T=M_wEM_u\in \mathcal{B}(L^2(\mu))$. Then the following conditions are mutually equivalent:

\begin{description}
\item[i]  $T$ is $1$-quasi-contraction;
\item[ii] $T$ is $n$-quasi-contraction, for every $n\in \mathbb{N}$;
\item[iii] There exists $n\in \mathbb{N}$ such that $T$ is $n$-quasi-contraction;
\item[iv] $|E(uw)|\leq1$, a.e., on $F=S(E(uw)$;
\item[v] $T$ is power bounded;
\item[vi] $\sigma_{T}\subseteq D$.
\end{description}
In which $D$ is closed unit disk and $\sigma_{T}$ is the spectrum of $T$.
  \end{thm}
  \begin{proof}
By Theorems \ref{t3.2} and \ref{t3.3} we have the equivalence of conditions (i), (ii), (iii), (iv) and (v). As we know from \cite{e1},  $\sigma(M_wEM_u)\setminus \{0\}=ess \
range(E(uw))\setminus\{0\}$ and so $r(T)=\|E(uw)\|_{\infty}$. So we easily get that (iv) and (vi) are equivalent. This completes the proof.
  \end{proof}
Let $X, Y$ be a Banach spaces and $T_n,T\in \mathcal{B}(X)$, for $n\in \mathbb{N}$. We say the sequence $\{T_n\}_{n\in \mathbb{N}}$ converges uniformly to $T$ if
$$\|T_n-T\|\rightarrow 0, \ \ \ n\rightarrow \infty.$$
Also, we say the sequence $\{T_n\}_{n\in \mathbb{N}}$ converges strongly to $T$ if
$$\|T_nx-Tx\|\rightarrow 0, \ \ \ n\rightarrow \infty, \ \ \ \ \ \text{for all} \ \ x\in X.$$
Moreover, we say the sequence $\{T_n\}_{n\in \mathbb{N}}$ converges weakly to $T$ if
$$|f(T_nx)-f(Tx)|\rightarrow 0, \ \ \ n\rightarrow \infty, \ \ \ \ \ \text{for all} \ \ x\in X, f\in Y^*.$$
By Banach-Steinhaus Theorem we get that if $\{T_n\}_{n\in \mathbb{N}}$ converges weakly to $T$, then $\sup_n\|T_n\|<\infty$.\\
It is clear that
$$\|T_n(x)-T(x)\|\leq \|T_n-T\|\|x\|, \ \ \ \ |f(T_nx)-f(Tx)|\leq \|f\|\|T_nx-Tx\|,$$
for all $ x\in X, f\in Y^*$. This means that uniformly convergence implies strongly convergence and strongly convergence implies weakly convergence. \\
Let $T\in \mathcal{B}(X)=\mathcal{B}(X,X)$. The operator $T$ is called uniformly stable, if the power sequence $\{T^n\}_{n\in \mathbb{N}}$ converges uniformly to the null operator; that is $$\|T^n\|\rightarrow 0$$
and also $T$ is called strongly stable if $\{T^n\}_{n\in \mathbb{N}}$ converges strongly to the null operator; that is
$$\|T^nx\|\rightarrow 0, \ \ \ \text{for all} \ \ \ x\in X.$$
Moreover, the operator $T$ is called weakly stable if the sequence $\{T^n\}_{n\in \mathbb{N}}$ is weakly convergent to the null operator; that is
$$|f(T^nx)|\rightarrow 0, \ \ \ n\rightarrow \infty, \ \ \ \ \ \text{for all} \ \ x\in X, f\in X^*.$$
Hence by Banach-Steinhaus Theorem we get that if $T$ is weakly stable, then it is power bounded. By the above observations we get that uniformly stability implies strongly stability and strongly stability implies weakly stability.\\

 Now in the next Theorem for WCT operator $T=M_wEM_u$ on the Banach space $L^p(\mu)$ we have $T$ is uniformly stable iff $T$ is strongly stable iff $T$ is weakly stable iff $\|E(uw)\|_{\infty}<1$  iff  $|E(uw)|<1$, a.e., $\mu$.
\begin{thm}\label{t2.4}
Let $T=M_wEM_u\in \mathcal{B}(L^p(\mu))$. Then $T$ is uniformly stable iff $T$ is strongly stable iff $T$ is weakly stable iff $|E(uw)|<1$, a.e., $\mu$.
\end{thm}
\begin{proof}
It is clear that if $T$ is uniformly stable, then it is strongly stable also if $T$ is strongly stable, then it is weakly stable. In addition if $T$ is weakly stable, then it is power bounded and so by Theorem \ref{t3.2} we have $|E(uw)|\leq 1$, a.e., $\mu$. We claim that we should have $|E(uw)|< 1$, a.e., $\mu$. Because, if $|E(uw)|=1$, a.e., $\mu$ on $A$ with $0<\mu(A)<\infty$, then for each $n\in \mathbb{N}$, we have $|E(uw)^n|=1$, a.e., $\mu$ on $A$. We can find $f\in L^p(\mu)$ and $g\in L^{q}$,
$$|\int_XT^n(f)gd\mu|\nrightarrow 0, \ \ \ \ \text{as} \ \ \ n\rightarrow \infty,$$
 where $\frac{1}{p}+\frac{1}{q}=1$. So $T$ is not weakly stable. This is a contradiction.

For the converse, let $n\in \mathbb{N}$ and $f\in L^p(\mu)$. Then we have $T^nf=E(uw)^{n-1}Tf$ and so
$$\|T^n\|=\sup_{\|f\|_p\leq 1}\|E(uw)^{n-1}Tf\|_p\leq \|E(uw)^{n-1}\|_{\infty}\|T\|.$$
By these observations if $|E(uw)|<1$, a.e., $\mu$, then $\|E(uw)^{n-1}\|_{\infty}\rightarrow 0$ as $n\rightarrow \infty$. Therefore $\|T^n\|\rightarrow 0$ as $n\rightarrow \infty$. Hence $T$ is uniformly stable and the proof is complete.
\end{proof}
A bounded linear operator on the Banach space $X$ is called convergent if
 $$\sigma_T\subseteq \{\lambda\in \mathbb{C}: |\lambda|<1\}.$$
 Hence by the above observations we get that the WCT operator $T=M_wEM_u\in \mathcal{B}(L^p(\mu))$ is convergent iff $ess \ range(E(uw))\subseteq \{\lambda\in \mathbb{C}: |\lambda|<1\}$ iff $|E(uw)|<1$, a.e., $\mu$. By these observations we have the following results.
 \begin{thm}\label{t2.5}
 The WCT operator $T=M_wEM_u\in \mathcal{B}(L^p(\mu))$ is convergent iff $ess \ range(E(uw))\subseteq \{\lambda\in \mathbb{C}: |\lambda|<1\}$ iff $|E(uw)|<1$, a.e., $\mu$.
  \end{thm}
  \begin{proof} As we know from Theorem 2.8 of \cite{e1}, the spectrum of $T$ is as
  $$\sigma(T)\cup\{0\}=ess \ range(E(uw))\cup\{0\}.$$
Since $\sigma(T)$ is compact, then we have the result.
  \end{proof}
  Now by Theorems \ref{t2.4} and \ref{t2.5} we have the following result.
  \begin{cor}\label{cor2.6}
  If WCT operator $T=M_wEM_u\in \mathcal{B}(L^p(\mu))$, then the following statements are equivalent:
  \begin{itemize}
    \item $T$ is convergent.
    \item $T$ is uniformly stable.
    \item $T$ is strongly stable.
    \item $T$ is weakly stable.
    \item $\|E(uw)\|_{\infty}<1$.
    \item $|E(uw)|<1$, a.e., $\mu$.
  \end{itemize}
  \end{cor}
$$\|T^{n+1}f-T^nf\|=\|M_{(E(uw)-1)}M_{E(uw)^{n-1}}Tf\|$$
$$\|T^n\|=\|M_{E(uw)^{n-1}}T\|\leq \|E(wu)\|^{n-1}_{\infty}\|T\|$$
It is clear that if $T\in \mathcal{B}(X)$, for a Banach space $X$, is power bounded, then $r(T)\leq 1$. But the converse is not true in general, i.e., $r(T)\leq 1$ doesn't imply power boundedness of $T$. To see this, let $T=\left[\begin{array}{cc}
                                                1 & 1 \\
                                                0 & 1
                                              \end{array}
                                              \right]
                                              $
 on $\mathbb{C}^2$. $r(T)\leq 1$, but $T$ is not power bonded.
 Here, as you see in the next Proposition, we have a large class of bounded linear operators that for each $T$ of them the power boundedness is equivalent to the condition $r(T)\leq 1$.
 \begin{prop}\label{t2.7}
 Let $T=M_wEM_u\in \mathcal{B}(L^p(\mu))$, \ \ $1\leq p\leq \infty$. Then $T$ is power bounded if and only if $r(T)\leq 1$.
 \end{prop}
 \begin{proof}
Since
$$\sigma(T)\cup\{0\}=ess \ range(E(uw))\cup\{0\},$$
then we have $r(T)=\|E(uw)\|_{\infty}$. Therefore by Theorem \ref{t3.2} we get that $T$ is power bounded if and only if $r(T)\leq 1$.
 \end{proof}

  In general $r(T)<1$ iff $\lim_{n\rightarrow \infty}\|T^n\|=0$ iff $T$ is uniformly exponentially stablei.e., there exist $M\geq 0$ and $\epsilon>0$ such that $\|T^n\|\leq Me^{-\epsilon n}$ for all $n\in \mathbb{N}$. If we apply it to WCT operators $T=M_wEM_u$ then we have $\|E(uw)\|_{\infty}<1$ iff $\lim_{n\rightarrow \infty}\||E(uw)|^{n-1}(E(|u|^{p'}))^{\frac{1}{p'}}(E(|w|^{p}))^{\frac{1}{p}}\|=0$ iff $T$ is uniformly exponentially stablei.e., there exist $M\geq 0$ and $\epsilon>0$ such that $\|T^n\|\leq Me^{-\epsilon n}$ for all $n\in \mathbb{N}$, in which $\frac{1}{p}+\frac{1}{p'}=1$.
  \\
  Since $T^*=M_{\bar{u}}EM_{\bar{w}}$, $\|T^n\|=\|T^{*^{n}}\|$, for each $n\in \mathbb{N}$ and $r(T^*)=r(T)=\|E(uw)\|_{\infty}$, then $T$ is power bounded if and only if $T^*$ is power bounded if and only if $|E(uw)|\leq1$, a.e., $\mu$. Moreover, we get that $T^*$ is convergent iff $T$ is convergent. Also, $T^*$ is stable (uniformly, strongly, weakly) iff $T$ is stable (uniformly, strongly, weakly) iff $|E(uw)|<1$, a.e., $\mu$.\\


Finally, we provide some concrete examples to illustrate our main results.
\begin{exam} (a) Let $X=\mathbb{N}\cup\{0\}$,
$\mathcal{G}=2^{\mathbb{N}}$ and let
$\mu(\{x\})=\frac{e^{-\theta}\theta^x}{x !}$, for each $x\in X$
and $\theta\geq0$. Elementary calculations show that $\mu$ is a
probability measure on $\mathcal{G}$, i.e., $\mu(X)=1$. Let $\mathcal{A}$ be the
$\sigma$-algebra generated by the partition $B=\{\emptyset, X,
\{0\}, X_1=\{1, 3, 5, 7, 9, ....\}, X_2=\{2, 4, 6, 8, ....\},\}$
of $\mathbb{N}$. Note that, $\mathcal{A}$ is a
sub-$\sigma$-algebra of $\Sigma$ and each of element of
$\mathcal{A}$ is an $\mathcal{A}$-atom. Thus, the conditional
expectation of any $f\in \mathcal{D}(E)$ relative to $\mathcal{A}$
is constant on $\mathcal{A}$-atoms. Hence there exists scalars
$a_1, a_2, a_3$ such that

$$E(f)=a_1\chi_{\{0\}}+a_2\chi_{X_1}+a_3\chi_{X_2}.$$
So
$$E(f)(0)=a_1, \ \ \ \ E(f)(2n-1)=a_2, \ \ \ \ E(f)(2n)=a_3,$$
for all $n\in \mathbb{N}$. By definition of conditional
expectation with respect to $\mathcal{A}$, we have
$$f(0)\mu(\{0\})=\int_{\{0\}}fd\mu=\int_{\{0\}}E(f)d\mu=a_1\mu(\{0\}),$$
so $a_1=f(0)$. Also,
$$\sum_{n\in \mathbb{N}}f(2n-1)\frac{e^{-\theta}\theta^{2n-1}}{(2n-1)
!}=\int_{X_1}fd\mu=\int_{X_1}E(f)d\mu$$$$=a_2\mu(X_2)=a_2\sum_{n\in
\mathbb{N}}\frac{e^{-\theta}\theta^{2n-1}}{(2n-1) !}.$$ and so

$$a_2=\frac{\sum_{n\in \mathbb{N}}f(2n-1)\frac{e^{-\theta}\theta^{2n-1}}{(2n-1)
!}}{\sum_{n\in \mathbb{N}}\frac{e^{-\theta}\theta^{2n-1}}{(2n-1)
!}}.$$

 By the same method we have

$$a_3=\frac{\sum_{n\in \mathbb{N}}f(2n)\frac{e^{-\theta}\theta^{2n}}{(2n)
!}}{\sum_{n\in \mathbb{N}}\frac{e^{-\theta}\theta^{2n}}{(2n)
!}}.$$

If $u$ and $w$ are real functions on $X$ such that $M_wEM_u$
is bounded on $l^p$, then we have
$$\sigma(M_wEM_u)=\{a_1=u(0)w(0),a_2=\frac{\sum_{n\in \mathbb{N}}u(2n-1)w(2n-1)\frac{e^{-\theta}\theta^{2n-1}}{(2n-1)
!}}{\sum_{n\in \mathbb{N}}\frac{e^{-\theta}\theta^{2n-1}}{(2n-1)
!}},a_3=\frac{\sum_{n\in
\mathbb{N}}u(2n)w(2n)\frac{e^{-\theta}\theta^{2n}}{(2n)
!}}{\sum_{n\in \mathbb{N}}\frac{e^{-\theta}\theta^{2n}}{(2n)
!}}\}.$$
Hence by Theorem \ref{t3.4} we ge that $T=M_wEM_u$ is power bounded iff $T$ is quasi-contraction iff $T$ is $n$-quasi-contraction, for every $n$ iff $\max{\{a_1,a_2,a_3\}}\leq 1$.\\
Also by Corollary \ref{cor2.6} we get that $T$ is convergent iff $T$ is uniformly stable iff $T$ is strongly stable iff $T$ is weakly stable iff $\max{\{a_1,a_2,a_3\}}<1$.\\

(b) Let $X=\mathbb{N}$, $\mathcal{G}=2^{\mathbb{N}}$ and let
$\mu(\{x\})=pq^{x-1}$, for each $x\in X$, $0\leq p\leq 1$ and
$q=1-p$. Elementary calculations show that $\mu(X)=1$ and so $\mu$ is a probability
measure on $\mathcal{G}$. Let $\mathcal{A}$ be the
$\sigma$-algebra generated by the partition
$B=\{X_1=\{3n:n\geq1\}, X^{c}_1\}$ of $X$. So, for every $f\in
\mathcal{D}(E^{\mathcal{A}})$,

$$E(f)=\alpha_1\chi_{X_1}+\alpha_2\chi_{X^c_1}$$
and direct computations show that

$$\alpha_1=\frac{\sum_{n\geq1}f(3n)pq^{3n-1}}{\sum_{n\geq1}pq^{3n-1}}$$
and
$$\alpha_2=\frac{\sum_{n\geq1}f(n)pq^{n-1}-\sum_{n\geq1}f(3n)pq^{3n-1}}{\sum_{n\geq1}pq^{n-1}-\sum_{n\geq1}pq^{3n-1}}.$$

So, if $u$ and $w$ are real functions on $X$ such that $M_wEM_u$
is bounded on $l^p$, then we have
$$\sigma(M_wEM_u)=\{\alpha_1=\frac{\sum_{n\geq1}u(3n)w(2n)pq^{3n-1}}{\sum_{n\geq1}pq^{3n-1}},
\alpha_2=\frac{\sum_{n\geq1}u(n)w(n)pq^{n-1}-\sum_{n\geq1}u(3n)w(3n)pq^{3n-1}}{\sum_{n\geq1}pq^{n-1}-\sum_{n\geq1}pq^{3n-1}}\}.$$\\

Hence by Theorem \ref{t3.4} we ge that $T=M_wEM_u$ is power bounded iff $T$ is quasi-contraction iff $T$ is $n$-quasi-contraction, for every $n$ iff $\max{\{\alpha_1,\alpha_2\}}\leq 1$.\\

Also by Corollary \ref{cor2.6} we get that $T$ is convergent iff $T$ is uniformly stable iff $T$ is strongly stable iff $T$ is weakly stable iff $\max{\{\alpha_1,\alpha_2\}}<1$.\\

(c) Let $X=[0,a]\times [0,a]$ for $a>0$, $d\mu=dxdy$, $\Sigma$ the
Lebesgue subsets of $X$ and let $\mathcal{A}=\{A\times [0,a]: A \
\mbox{is a Lebesgue set in} \ [0,a]\}$. Then, for each $f\in
\mathcal{D}(E)$, $(Ef)(x, y)=\int_0^af(x,t)dt$, which is
independent of the second coordinate. For example, if we set
$a=1$, $w(x,y)=1$ and $u(x,y)=e^{(x+y)}$, then
$E(u)(x,y)=e^x-e^{x+1}$ and $M_wEM_u$ is bounded. Therefore $\sigma(M_wEM_u)=[e-e^2, 1-e]$. Therefore we have $|E(uw)|>1$, for all $x,y$. Consequently by Theorem \ref{t3.4} we ge that $T=M_wEM_u$ is not power bounded, equivalently is not quasi-contraction and also is not $n$-quasi-contraction.
Also by Corollary \ref{cor2.6} we get that $T$ is not convergent, is not uniformly stable, is not strongly stable and also is not weakly stable.
\end{exam}

\textbf{Declarations}\\
     \textbf{Conflict of interest.} On behalf of all authors, the corresponding author states that there is no conflict of interest.\\
     \textbf{Acknowledgement.} Our manuscript has no associate data.

\end{document}